\documentclass[10pt,a4paper]{amsart}
\usepackage{graphicx,multirow,array,amsmath,amssymb}

\newtheorem{theorem}{Theorem}

\newtheorem{proposition}[theorem]{Proposition}
\newtheorem{lemma}[theorem]{Lemma}

\begin{document}

\title{Links with small lattice stick numbers}

\author[K. Hong]{Kyungpyo Hong}
\address{Department of Mathematics, Korea University, Anam-dong, Sungbuk-ku, Seoul 136-701, Korea}
\email{cguyhbjm@korea.ac.kr}
\author[S. No]{Sungjong No}
\address{Department of Mathematics, Korea University, Anam-dong, Sungbuk-ku, Seoul 136-701, Korea}
\email{blueface@korea.ac.kr}
\author[S. Oh]{Seungsang Oh}
\address{Department of Mathematics, Korea University, Anam-dong, Sungbuk-ku, Seoul 136-701, Korea}
\email{seungsang@korea.ac.kr}

\keywords{knot, stick number, cubic lattice}
\thanks{{\em PACS numbers\/}. 02.10.Kn, 82.35.Pq, 02.40.Sf}
\thanks{This research was supported by Basic Science Research Program through
the National Research Foundation of Korea(NRF) funded by the Ministry of Science,
ICT \& Future Planning(MSIP) (No.~2011-0021795).}

\begin{abstract}
Knots and links have been considered to be useful models for structural analysis of
molecular chains such as DNA and proteins.
One quantity that we are interested on molecular links is the minimum number of monomers
necessary to realize them.
In this paper we consider every link in the cubic lattice.
Lattice stick number $s_L(L)$ of a link $L$ is defined to be the minimal number of sticks
required to construct a polygonal representation of the link in the cubic lattice.
Huh and Oh found all knots whose lattice stick numbers are at most 14.
They proved that only the trefoil knot $3_1$ and the figure-8 knot $4_1$ have lattice stick
numbers 12 and 14, respectively.
In this paper we find all links with more than one component
whose lattice stick numbers are at most 14.
Indeed we prove combinatorically that $s_L(2^2_1)=8$,
$s_L(2^2_1 \sharp 2^2_1)=s_L(6^3_2)=s_L(6^3_3)=12$, $s_L(4^2_1)=13$, $s_L(5^2_1)=14$
and any other non-split links have stick numbers at least 15.
\end{abstract}

\maketitle

\section{Introduction} \label{sec:intro}

A link is a union of closed curves in $3$-space $\mathbb{R}^3$.
A knot is a link with one component.
Links are commonly found in molecular chains such as DNA and proteins,
and they have been considered to be useful models for simulating molecular chains.
A link can be embedded in many different ways in $3$-space, smooth or piecewise linear.
{\em Polygonal links\/} are those which consist of line segments, called {\em sticks\/},
attached end-to-end.
This representation of links can be considered to be a reasonable mathematical model
of cyclic molecules or molecular chains because such physical objects have rigidity.
It has many applications in natural science.
In microscopic-level molecules, DNA strands are made up of small rigid sticks of sugar,
phosphorus, nucleotide proteins and hydrogen bonds.
Chemists are also interested in knotted molecules which are formed by a sequence of atoms
bonded end-to-end so that the last one is also bonded to the first.

Concerning polygonal links, one natural problem may be to determine the number of sticks.
The {\em stick number\/} $s(L)$ of a link $L$ is defined to be the minimum number of sticks
required to construct a polygonal representation of the link.
This quantity was investigated for some specific knots including knots
with crossing number below 10 \cite{C,R}, torus knots \cite{J},
2-bridge knots \cite{Mc} and knots with 1, 2 and 3-integer Conway notations \cite{FLS}.
On the other hand Negami \cite{N} found a general upper bound in terms of crossing number $c(K)$
which is $s(K) \leq 2 c(K)$ for a nontrivial knot $K$.
Later Huh and Oh \cite{HO3} improved to $s(K) \leq \frac{3}{2} (c(K)+1)$,
and moreover $s(K) \leq \frac{3}{2} c(K)$ for a non-alternating prime knot.

From now on we deal with anther quantity of polygonal links.
A {\em lattice link\/} is a polygonal link in the cubic lattice $\mathbb{Z}^3$
which is $(\mathbb{R} \times \mathbb{Z} \times \mathbb{Z}) \cup (\mathbb{Z} \times \mathbb{R}
\times \mathbb{Z}) \cup (\mathbb{Z} \times \mathbb{Z} \times \mathbb{R})$.
The {\em lattice stick number\/} $s_L(L)$ of a link $L$ is defined to be the minimum number of sticks
required to construct a polygonal representation of the link in the cubic lattice $\mathbb{Z}^3$.
We may say that this quantity corresponds to the total curvature of smooth knots
which was firstly studied by Milnor \cite{Mi}.
He showed that the total curvature of any smooth non-trivial knot is at least $4 \pi$.
For a polygonal knot in $\mathbb{R}^3$ its total curvature can be defined as the summation of all angles
between every pair of adjacent sticks.
This quantity says how much the modeled polymer turn in space,
therefore is expected to have some connection with physical properties of molecular chains.
In \cite{P}, using numerical simulations, the total curvature of polygons was scaled
as a function of the length of polygons for each knot with up to six crossings.
Also it was reported that the equilibrium length with respect to total curvature,
which appears to be correlated to physical properties of macromolecules,
can be considered to be one of characteristics of each knot under the experiment.
In lattice links, each angle between any pair of adjacent sticks is $\frac{\pi}{2}$.
Hence the total curvature of any link $L$ can be clearly defined
to be $\frac{\pi}{2} s_{L}(L)$ with no necessity to consider its length.
Furthermore the restriction on the position of sticks establishes
some combinatorial arguments which may allow theoretical study on the quantity
other than computational simulations.

For the case of knots, first the reader easily knows that
the lattice stick number of the trivial knot is 4.
Janse van Rensburg and Promislow \cite{JP} proved that $s_L(K) \geq 12$ for any nontrivial knot $K$.
Also they estimated the quantity for various knots via the simulated annealing technique.
Another numerical estimation was performed in \cite{SIADSV}.
On the other hand Huh and Oh \cite{HO1,HO2} proved combinatorically that
$s_L(3_1) = 12$, $s_L(4_1) = 14$ and $s_L(K) \geq 15$ for any other non-trivial knot $K$.
They used an elementary but useful argument, called {\em proper levelness\/}.
Recently Adams et al. \cite{ACCJSZ} announced that $s_L(8_{20}) = s_L(8_{21}) = s_L(9_{46}) = 18$,
$s_L(4^2_1) = 13$ and $s_L(T_{p,p+1}) = 6p$ for $p \geq 2$ where $T_{p,p+1}$ is a $(p,p+1)$-torus knot.
To find the exact values of the lattice stick number of these knots,
they used a lower bound on lattice stick number in terms of bridge number, $s_L(K) \geq 6 b(K)$,
which was proved in \cite{JP}.
Furthermore Diao and Ernst \cite{DE} found a general lower bound in terms of crossing number $c(K)$
which is $s_L(K) \geq 3 \sqrt{c(K)+1} +3$ for a nontrivial knot $K$.
Recently Hong, No and Oh \cite{HNO} found a general upper bound $s_L(K) \leq 3 c(K) +2$,
and moreover $s_L(K) \leq 3 c(K) - 4$ for a non-alternating prime knot.

In this paper we find all links with more than one component whose lattice stick numbers are at most 14.
In the proof we use the {\em extended proper levelness\/}
which is a modified version of the proper levelness of knots to links.
The following theorem is the main result.

\begin{theorem} \label{thm:main}
There are only 6 links (with more than one component) whose lattice stick number is less than or equal to 14.
These links are $2_1^2$, $2_1^2 \sharp 2_1^2$, $6_2^3$, $6_3^3$, $4_1^2$ and $5_1^2$.
Furthermore, it is known that $s_L(2_1^2)=8$, $s_L(2_1^2 \sharp 2_1^2)=s_L(6_2^3)=s_L(6_3^3)=12$,
$s_L(4_1^2)=13$ and we have further shown here that $s_L(5_1^2)=14$.
\end{theorem}

Note that the existences of six lattice links representing their lattice stick numbers
as depicted in Figure \ref{fig1} guarantee their upper bounds.

\begin{figure}[ht]
\includegraphics{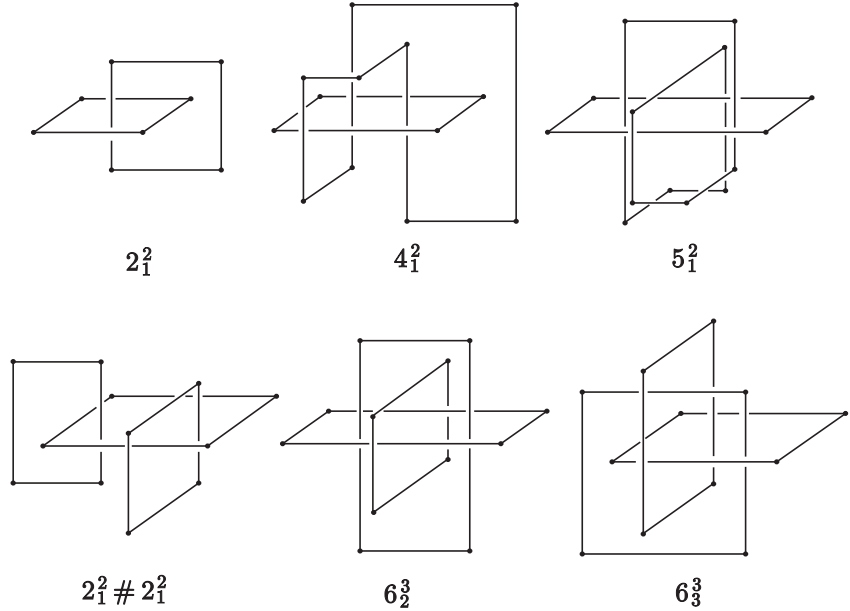}
\caption{Six examples of lattice links}
\label{fig1}
\end{figure}

\section{Extended Proper levelness}

In this section we introduce some definitions and a lemma which are found at Section 2 in \cite{HO1}.
Two lattice links are said to be {\em equivalent\/} if they are ambient isotopic in $\mathbb{R}^3$.
A lattice link $L$ is called {\em reducible\/}
if there is another equivalent lattice link which has fewer sticks.
Otherwise, it is called {\em irreducible\/}.

From now on the notations concerning the $y$ and $z$-coordinates will be defined
in the same manner as the $x$-coordinate.
For an $n$-component lattice link $L$, let $\{L^1, \cdots, L^n\}$ be the set of $n$ components.
$|L|$ denotes the number of sticks of $L$.
A stick in $L$ which is parallel to the $x$-axis is called an {\em $x$-stick\/},
and $|L|_x$ denotes the number of its $x$-sticks.
Let us denote $(L)=(|L|_x,|L|_y,|L|_z)$.
An {\em $x$-level\/} $k$ for some integer $k$ is a $yz$-plane whose $x$-coordinate is $k$.
Then each $y$-stick and $z$-stick lies entirely on an $x$-level.
If all $y$-sticks and $z$-sticks of $L$ are contained in exactly $n$ $x$-levels,
then, without loss of generality,
we may say that these are $x$-levels $1,2,\cdots,n$ like height numbers.
Note that an $x$-stick whose endpoints lie on $x$-levels $i$ and $j$ has length $|i-j|$.

A lattice link $L$ is said to be {\em properly leveled with respect to the $x$-coordinate\/}
if each $x$-level contains exactly two endpoints of $x$-sticks,
so all $y$-sticks and $z$-sticks on an $x$-level are connected as an arc.
$L$ is said to be {\em properly leveled\/} if it is properly leveled with respect to each coordinate.
Furthermore $L$ is said to be {\em extended properly leveled with respect to the $x$-coordinate\/}
if each $x$-level contains
either exactly two endpoints of $x$-sticks or exactly one component $L^i$ of $L$,
and {\em extended properly leveled\/} if it is extended properly leveled with respect to each coordinate.

\begin{lemma} \label{lem:epl}
For a lattice link $L$, there is an extended properly leveled lattice link $L'$ equivalent to $L$
so that $|L'| = |L|$ (indeed, the numbers of $x$, $y$ and $z$-sticks have remained unchanged).
\end{lemma}

\begin{proof}
This proof follows the proof of Lemma 2.1. in \cite{HO1}.
Obviously the part of $L$ at each $z$-level is the disjoint union of some components of $L$ or
some parts of $L^i$'s.
Assume that, at $z$-level $m$, it consists of $p$ entire components $L^i$'s and
$q$ connected parts of some components.
Each of $p+q$ portions consists of $x$-sticks and $y$-sticks.
We will insert $p+q-1$ more $z$-levels between $z$-levels $m$ and $m+1$.
This means that $z$-level $m+1$ goes to $z$-level $m+p+q$.
Now move up $p+q$ portions so that each $z$-level from $z$-level $m$ to $z$-level $m+p+q-1$
contains exactly one portion of them.
Also extend the related $z$-sticks so that their endpoints are attached properly to the portions used above.
Repeat this operation at every $z$-level.
Then each $z$-level contains either exactly two endpoints of $x$-sticks or exactly one component of $L$.
So the resulting lattice link is extended properly leveled with respect to the $z$-coordinate.
In order to get a desired extended properly leveled lattice link $L'$,
repeat the same arguments regarding the $x, y$-coordinates.
One notices that $L'$ is equivalent to $L$
and the numbers of $x$, $y$ and $z$-sticks have remained unchanged.
\end{proof}

\section{Proof of Theorem \ref{thm:main}}

In this section we will prove Theorem \ref{thm:main}.
Let $L=\{L^1,\cdots L^n\}$ be a non-split irreducible $n$-component lattice link with $n \geq 2$.
Furthermore we assume that $|L| \leq 14$.
We may assume that $L$ is extended properly leveled by Lemma \ref{lem:epl}.

We can easily show the followings for each component $L^i$ and each coordinate, say $x$;
$|L^i|_x \neq 1$, $|L^i|_x \leq |L^i|/2$ and also $|L^i| = 4$ or $|L^i| \geq 6$.
Thus we can easily prove that $s_L(2^2_1)=8$ and $s_L(2^2_1 \sharp 2^2_1)=s_L(6^3_2)=s_L(6^3_3)=12$.
From now on we assume that $L$ is neither $2^2_1$ nor $2^2_1 \sharp 2^2_1$.

Theorem 1.1 in \cite{HO1} says that $s_L(K) \geq 12$ for a non-trivial knot $K$.
Therefore $L$ must consist of only two or three trivial knot components.
$L^i$ is called {\em planar\/} if it is contained in a plane and let $P^i$ be this plane.
Let $P_b^i$ and $P_u^i$ be the bounded component and the unbounded component
of $P^i \setminus L^i$, respectively.
Now we introduce a useful proposition.

\begin{proposition} \label{prop:4pt}
Suppose that $L$ is a non-split link with at most 14 sticks
which is neither $2^2_1$ nor $2^2_1 \sharp 2^2_1$.
Suppose that $L^i$ is planar for some $i$ where $P^i$ is perpendicular to the $z$-axis, for example.
Then $L$ has exactly four $z$-sticks,
and furthermore two $z$-sticks pass through $P_b^i$ and the other two $z$-sticks pass through $P_u^i$.
\end{proposition}

\begin{proof}
If none of $z$-sticks of $L$ passes through $P_b^i$,
then we can apart $L^i$ from the other components of $L$.
Thus $L$ is a split lattice link.
If only one $z$-stick of $L$ passes through $P_b^i$,
then we can shrink $L^i$ into a small neighborhood of the passing point.
This means that $L$ is $2_1^2$ or has a $2_1^2$ connected summand.
In the latter case $|L \setminus L^i| \leq 10$ which consists of two trivial knot components.
Thus it must be either $0_1^2$ or $2_1^2$, a contradiction.
Thus at least two $z$-sticks pass through $P_b^i$ (similarly for $P_u^i$),
and so at least four $z$-sticks pass through $P^i$.

Furthermore if $L$ has six or more $z$-sticks,
then $L \setminus L^i$ has at least twelve sticks, contradicting $|L| \leq 14$.
If $L$ has five $z$-sticks, then one of them, say $z_1$, cannot meet $P^i$.
Let $a$ and $b$ be the $z$-levels of the endpoints of $z_1$
where $a < b$ and they are above, say, the $z$-level of $P^i$.
Let $w_a$ and $w_b$ be the $x$- or $y$-sticks attached to $z_1$
at $z$-levels $a$ and $b$, respectively.
Then there must be other two $x$- or $y$-sticks so that
one stick lies above $w_a$ and one stick lies between $w_b$ and $P^i$.
Otherwise, $L$ is reducible by shrinking the stick $z_1$.
Thus $L$ has at least four $x$- or $y$-sticks above $P^i$ and
at least two $x$- or $y$-sticks below $P^i$.
This implies that $|L| \geq 15$.
This completes the proof.
\end{proof}

Now we consider $L$ in the following three different cases;

\begin{itemize}
\item $L$ consists of only non-planar components.
\item $L$ contains exactly one planar component.
\item $L$ contains at least two planar components.
\end{itemize}
\vspace{2mm}

\noindent {\bf Case  1.} $L$ consists of only non-planar components.

Then $L$ is a 2-component link whose components are $L^1$ and $L^2$.
Each component contains at least two $x$-sticks, two $y$-sticks and two $z$-sticks
since both are non-planar.
We may say that $(L)$ is one of $(4,4,4)$, $(4,4,5)$, $(4,4,6)$ or $(4,5,5)$.

Now consider the projection of $L$ onto an $xy$-plane.
If $L^1$ has exactly two $x$-sticks and two $y$-sticks,
then its projection must be one of the left two diagrams in the first row of Figure \ref{fig2}.
If $L^1$ has exactly two $x$-sticks and three $y$-sticks,
then its projection is one of the other diagrams in the figure.
Note that we draw the projections where some sticks are perturbed slightly
just for better view, but these do not mean real.
We do similarly for $L^2$.

\begin{figure}[ht]
\includegraphics{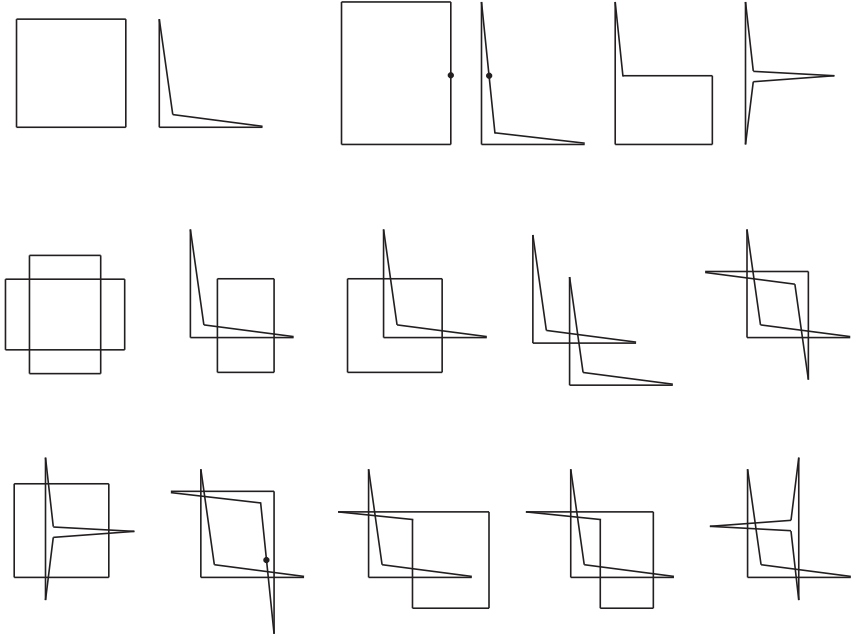}
\caption{Two non-planar components}
\label{fig2}
\end{figure}

First, consider that $|L^1|_x=|L^1|_y=|L^2|_x=|L^2|_y=2$.
Then all possible projections of $L$ are illustrated in the second row of Figure \ref{fig2}.
Indeed we ignored the projections with 0 or 2 crossings since $L$ is non-split and not $2^2_1$.
Each of the first four projections has four crossings,
but we can not make it an alternating diagram in any choices of positions of $z$-sticks.
This implies that $L$ must be split or $2^2_1$.
The last projection has eight crossings.
In any choices of positions of $z$-sticks, $L$ is one of a split link, $2^2_1$ or $4^2_1$.
To construct $4^2_1$ from this projection, six $z$-sticks are needed,
so fourteen sticks are needed in total, contradicting $s_L(4^2_1) \leq 13$.

Now assume that $|L^1|_x=|L^1|_y=|L^2|_x=2$ and $|L^2|_y=3$.
This is the case that $(L)=(4,5,5)$, so $|L|=14$.
It is sufficient to check the projections with more than five crossings.
Such projections are illustrated in the third row of Figure \ref{fig2}.
For the first or the third projection
we can not make it an alternating diagram in any choices of positions of $z$-sticks, so $c(L) < 6$.
For the second projection, $L$ is reducible.
The last two projections have eight crossings.
But in any choices of positions of five $z$-sticks, one can figure out that $c(L) \leq 5$.
\vspace{3mm}

\noindent {\bf Case 2.} $L$ contains exactly one planar component, say $L^1$.

Obviously $L$ is a 2-component link and let $L^2$ be the other non-planar component.
Assume that $P^1$ is perpendicular to the $z$-axis.
Proposition \ref{prop:4pt} guarantees that $L^2$ has exactly four $z$-sticks
and only two sticks pass through $P^1_b$.
Also $L^2$ has at least two $x$-sticks and $y$-sticks,
and thus $L^1$ consists of four or six sticks.

First, consider the case that $|L| \leq 12$.
Then $L^1$ consists of two $x$-sticks and two $y$-sticks,
and $L^2$ consists of four $z$-sticks, two $x$-sticks and two $y$-sticks.
Now consider the projection of $L$ onto an $xy$-plane.
The projection of $L^2$ must be rectangular
because each pair of an $x$-stick and a $y$-stick is connected by a $z$-stick
and these four $z$-sticks can not be overlapped in the projection.
This implies that $c(L) \leq 2$.

Second, consider the case that $|L| = 13$.
Assume for contradiction that $c(L) \geq 5$.
Similarly $L^1$ consists of two $x$-sticks and two $y$-sticks,
and $L^2$ consists of four $z$-sticks, two $x$-sticks and three $y$-sticks
(or three $x$-sticks and two $y$-sticks).
Thus the projection of $L^2$ is one of the right four diagrams in the first row of Figure \ref{fig2}.
There is only one choice of combinations of $L^1$ and $L^2$ so that $c(L) \geq 5$
which is illustrated by the leftmost figure in the third row of Figure \ref{fig2}.
But this contradicts to Proposition \ref{prop:4pt} because three $z$-sticks pass through $P^1_u$.
Therefore $c(L) \leq 4$ and $4^2_1$ is the only one link whose lattice stick number is 13.

Finally consider the case that $|L| = 14$.
Assume for contradiction that $c(L) \geq 6$.
When $L^1$ consists of six sticks such as three $x$-sticks and three $y$-sticks,
$L^2$ so consists of four $z$-sticks, two $x$-sticks and two $y$-sticks.
Then $L^1$ is an `L' shaped planar six-gon and the projection of $L^2$ is a rectangle.
This implies that $c(L) \leq 4$.
So we assume that $L^1$ consists of two $x$-sticks and two $y$-sticks.
When $L^2$ consists of four $z$-sticks, two $x$-sticks and four $y$-sticks
(or four $x$-sticks and two $y$-sticks),
there are four $y$-levels for $L^2$ and the three types of its four $y$-sticks are illustrated
as the left three figures in the first row of Figure \ref{fig3}.
And its two $x$-sticks can be located in six different ways considering symmetry
as in the right six figures.

\begin{figure}[ht]
\includegraphics{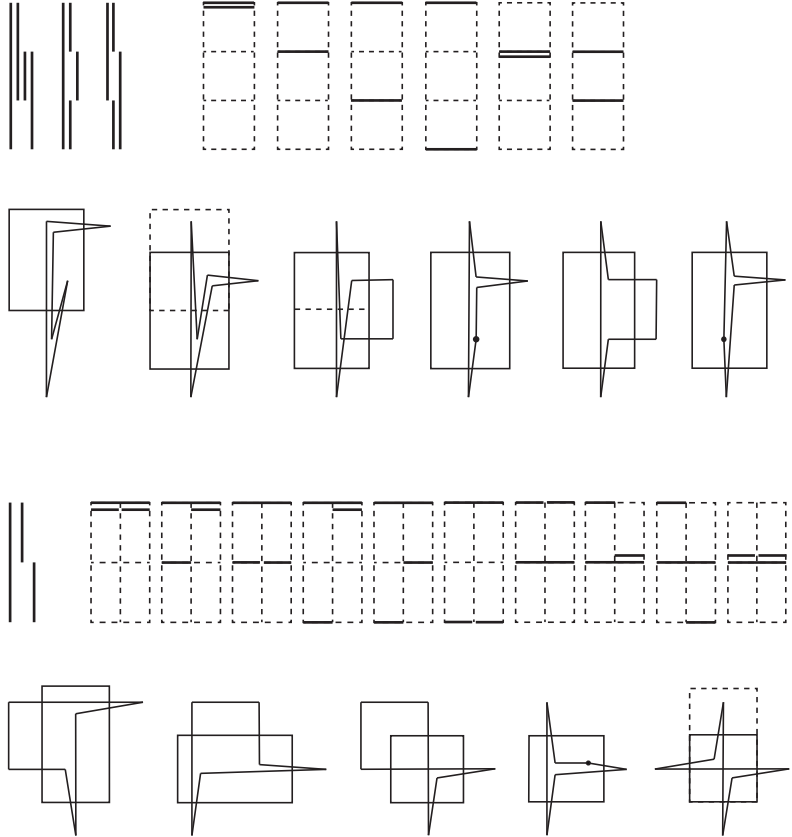}
\caption{one planar and one non-planar components}
\label{fig3}
\end{figure}

Among all 18 combinations only six diagrams illustrated in the second row
can have at least 6 crossings, but all the related $L^2$ pass through $P^1_u$
at least three points, a contradiction.
Note that in the third and fifth diagrams the rightmost $y$-stick of $L^2$ must be connected
to a $z$-stick, for otherwise we can reduce 2 crossings by a Reidermeister II move.
When $L^2$ consists of four $z$-sticks, three $x$-sticks and three $y$-sticks,
there are three $y$-levels for $L^2$ and its three $y$-sticks are illustrated
as the leftmost figure in the third row of Figure \ref{fig3}.
And its three $x$-sticks can be located in ten different ways considering symmetry
as in the right ten figures.
Among all possible combinations only five diagrams illustrated in the fourth row
can have at least 6 crossings, but all the related $L^2$ pass through $P^1_u$
at least three points as in the previous case, a contradiction.
Therefore $c(L) \leq 5$ and $5^2_1$ is the only one link whose lattice stick number is 14.
\vspace{3mm}

\noindent {\bf Case 3.} $L$ contains at least two planar components, say $L^1$ and $L^2$.

If $L$ is a 2-component link,
then Proposition \ref{prop:4pt} guarantees that
two planes $P^1$ and $P^2$ are perpendicular to each other.
Without loss of generality, we may say that $P^1$ is perpendicular to the $z$-axis
and $P^2$ is perpendicular to the $y$-axis.
Again Proposition \ref{prop:4pt} implies that $L^1$ has four $y$-sticks
and $L^2$ has four $z$-sticks.
Thus $|L^1| \geq 8$ and $|L^2| \geq 8$, i.e. $|L| \geq 16$, a contradiction.

Thus $L$ is a 3-component link and let $L^3$ be the third component.
Since each component consists of either four sticks or at least six sticks,
we may assume that $L^1$ and $L^2$ consist of four sticks respectively
and $L^3$ consists of four or six sticks.
By Proposition \ref{prop:4pt}, $L^1$ must pass through $P^2$ at two points
and similarly $L^2$ must pass through $P^1$ at two points.
Otherwise, if $L^3$ passes through $P^2$ at four points, then $|L^3| \geq 8$.
We may say that $P^1$ is perpendicular to the $z$-axis
and $P^2$ is perpendicular to the $y$-axis.

First assume that $P_b^1$ and $P_b^2$ do not meet each other.
Since $L^1$ do not pass through $P_b^2$ and $L^2$ do not pass through $P_b^1$,
$L^3$ must passes through $P_b^1$ at two points and $P_b^2$ at two points.
Then the two $z$-sticks of $L^3$ passing through $P_b^1$ must be connected to
the two $y$-sticks of $L^3$ passing through $P_b^2$ as in Figure \ref{fig4}(a).
Then, obviously $L$ is reducible.

\begin{figure}[ht]
\includegraphics{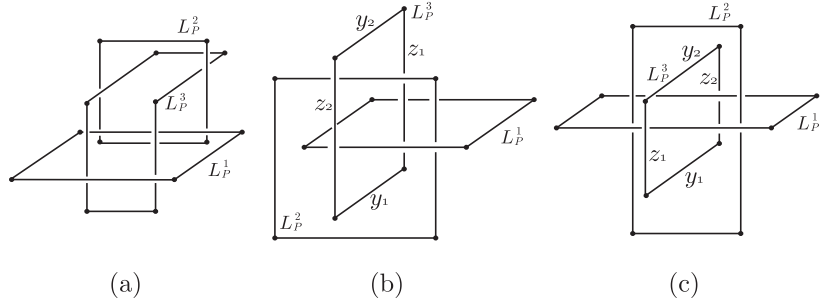}
\caption{3-component link cases}
\label{fig4}
\end{figure}

Next assume that $L^2$ passes through $P_b^1$ at one point.
There are two $z$-sticks of $L^3$ passing through $P^1$ and
let $z_1$ and $z_2$ be the $z$-sicks passing through $P_b^1$ and $P_u^1$, respectively.
Similarly let $y_1$ and $y_2$ be two $y$-sicks of $L^3$
passing through $P_b^2$ and $P_u^2$, respectively.
If $L^3$ consists of four sticks,
then $z_1$, $y_1$, $z_2$ and $y_2$ are consecutive in this order,
so obviously $L$ is $6^3_3$ as in Figure \ref{fig4}(b).
If $L^3$ consists of six sticks, for any pair of $z_i$ and $y_j$, $i,j = 1,2$,
$z_i$ and $y_j$ is adjacent or connected by at most two other sticks.
But these extra sticks can not pass through $P^1$ and $P^2$,
so $L$ is $6^3_3$ again.

Finally assume that $L^2$ passes through $P_b^1$ at two points
(or similarly $L^1$ passes through $P_b^2$ at two points).
Since the two $y$ sticks of $L^1$ passing through $P_u^2$,
$L^3$ must have two $y$-sticks $y_1$ and $y_2$ passing through $P_b^2$.
Similarly $L^3$ also has two $z$-sticks $z_1$ and $z_2$ passing through $P_u^1$.
First, consider the case that $y_1$ and $y_2$ lie on the other sides of the plane $P^1$.
If $L^3$ consists of four sticks,
then $z_1$, $y_1$, $z_2$ and $y_2$ are consecutive in this order,
so obviously $L$ is $6^3_2$ as in Figure \ref{fig4}(c).
If $L^3$ consists of six sticks, for any pair of $z_i$ and $y_j$, $i,j = 1,2$,
$z_i$ and $y_j$ is adjacent or connected by at most two other sticks.
But these extra sticks can not pass through $P^1$ and $P^2$,
so $L$ is $6^3_2$ again.
Now assume that $y_1$ and $y_2$ lie on the same side of $P^1$, say above $P^1$.
Then no stick passes through the component of $P_b^2 \setminus P_b^1$ which is below $P^1$.
Also only two $z$-sticks of $L^2$ which are parts of the boundary of $P_b^2$ pass through $P_b^1$.
This implies that $L^1$ can be split apart from the other components of $L$ by an isotope,
so $L$ is a split link, a contradiction.

This completes the proof of Theorem \ref{thm:main}.

\section{Conclusion}

In this paper we consider links with more than one component with small crossings.
The proper levelness was very useful to put especially knots in the cubic lattice properly.
But in the case of links, it doesn't work any more.
So we propose a modified version, so-called the extended proper levelness.
This give us the way to put links in the cubic lattice properly as we want.
Unlike knots, there are six links have the lattice stick number at most 14.
In this paper the lattice stick numbers of all six links were found.
For future work, we try to find all links with the lattice stick number 15.
In this case we have to handle more diagrams in each case.
Lastly, we want to point out that our results are obtained by
purely analytic calculations instead of numerical calculations.

\end{document}